\documentclass[a4paper,12pt]{article}

\usepackage[top=2.00cm,bottom=2.00cm,left=2.00cm,right=2.00cm]{geometry}%{geometry} ±íê??éò?×??¨ò?ò3??éè??
\usepackage{amsmath,amsthm,mathrsfs,graphicx,amsfonts}%{mathrsfs}ó?óú2úéúò???êy?§ó?μ??¨ì?×? ￡?êy?§1?ê?oê°ü
\usepackage{bm}%′|àíêy?§1?ê??Dμ?oúD±ì? ￡?1?ê??Dμ?′?ì?×?·? ￡¨ó??üá?\boldsymbol ￡?
\usepackage{cases}%\begin{numcases}{|x|=}x,&for$x\geq0$\\-x,&for$x<0$\end{numcases}
\usepackage[english]{babel}%??·¨
\usepackage{amsmath,amsthm}
\usepackage{amsfonts}
\usepackage{latexsym}
\usepackage{graphicx}%í?D?oê°ü
\usepackage{txfonts}
\usepackage[numbers,sort&compress]{natbib}
\usepackage[natural]{xcolor}
\usepackage{rotating}
\usepackage{mathtools}
\usepackage{enumerate}
\allowdisplaybreaks[4]

\newtheorem{lem}{Lemma}[section]
\newtheorem{thm}[lem]{Theorem}

\newtheorem{claim}{\indent Claim}[lem]
\newtheorem{definition}[lem]{Definition}

\begin{document}
\title{Extremal graphs for the suspension of edge-critical graphs\thanks{ Research partially supported by NSFC (Grant No. 12071077,  12001106) and National Natural Science Foundation of Fujian Province (Grant No. 2021J05128).}}
\author{Jianfeng Hou, Heng Li, Qinghou Zeng\\
{\small Center for Discrete Mathematics, Fuzhou University, Fujian, 350003, China}}

\date{}

\maketitle

\begin{abstract}
The Tur\'{a}n number of a graph $H$, $\text{ex}(n,H)$, is the maximum number of edges in an $n$-vertex graph that does not contain $H$ as a subgraph. For a vertex $v$ and a multi-set $\mathcal{F}$ of graphs, the suspension $\mathcal{F}+v$ of $\mathcal{F}$ is the graph obtained by connecting the vertex $v$ to all vertices of $F$ for each $F\in \mathcal{F}$. For two integers $k\ge1$ and $r\ge2$, let $H_i$ be a graph containing a critical edge with chromatic number $r$ for any $i\in\{1,\ldots,k\}$, and let $H=\{H_1,\ldots,H_k\}+v$. In this paper, we determine $\text{ex}(n, H)$ and characterize all the extremal graphs for sufficiently large $n$. This generalizes a result of Chen, Gould, Pfender and Wei on intersecting cliques. We also obtain a stability theorem for $H$, extending a result of Roberts and Scott on graphs containing a critical edge.
\end{abstract}

\textbf{Keywords:} Tur\'{a}n number, extremal graph, edge-critical graph, $r$-partite

\section{Introduction}\label{Intro}
Given a graph $H$, a graph $G$ is called $H$-\emph{free} if it contains no copy of $H$ as a subgraph. The \emph{Tur\'{a}n number} ${\rm ex}(n, H)$ of $H$ is the maximum number of edges in an $H$-free graph on $n$ vertices. Determining ${\rm ex}(n, H)$ is one of most important problem in extremal graph theory and the \emph{Tur\'{a}n graph} plays a key role. For two integers $n$ and $r$ with $n\ge r\ge2$,  the Tur\'{a}n graph $T_r(n)$ is an $n$-vertex complete $r$-partite graph with parts of size $\lceil n/r\rceil$ or $\lfloor n/r\rfloor$. Let $t_r(n)$ denote the number of edges in $T_r(n)$. %Any $H$-free graph with $n$ vertices and  ${\rm ex}(n, H)$ edges is called an extremal graph for $H$.
The classical  Tur\'{a}n's Theorem \cite{turan1941}  shows that ${\rm ex}(n, K_{r+1})=t_r(n)=(1-\frac{1}{r}+o(1)){n \choose 2}$ and  the only extremal graph is $T_r(n)$.

Let $\chi(H)$ denote the chromatic number of $H$. If there is an edge $e$ of $H$ such that $\chi(H-e)=\chi(H)-1$, then we say that $H$ is \emph{edge-critical} and $e$ is a \emph{critical edge}. The celebrated Erd\H{o}s-Stone-Simonovits Theorem \cite{ES1966,Stone1946} states that ${\rm ex}(n, H)=\left(1-\frac{1}{\chi(H)-1}+o(1)\right){n \choose 2}$. For an edge-critical graph $H$ with $\chi(H)=r+1$, Simonovits \cite{Simonovits1974} proved that $T_r(n)$ is also the unique extremal graph for sufficiently large $n$.% We call an edge $e\in E(H)$ a critical edge of $H$ if removal $e$ from $H$ will decrease the chromatic number of $H$. In 1974, Simonovits \cite{Simonovits1974} generalized Tur\'{a}n's theorem to edge-critical graph by proving the following theorem.
\begin{thm}[Simonovits \cite{Simonovits1974}]\label{Simonovits}
Let $H$ be an edge-critical graph with $\chi(H)=r+1\geq 3$. Then there exists some $n_0$ such that ${\rm ex}(n, H) = t_r(n)$ for all $n \geq n_0$, and the unique extremal graph is $T_r(n)$.
\end{thm}

Although the Tur\'an numbers of non-bipartite graphs are asymptotically determined by Erd\H{o}s-Stone-Simonovits theorem, it is still a challenge to determine the exact Tur\'an numbers for many non-bipartite graphs. There are only a few graphs whose Tur\'an numbers are determined exactly, including edge-critical graphs and some other specific graphs (e.g. see \cite{FUREDI2015, He-Ma-Yang2021, LAN2019, llp2013, yuan2021, yuan2022, yuanlong2022, yuanzhang2022}). Among all the existing results, the Tur\'an number of the graph consisting of some specific graphs that intersect in exactly one common vertex is widely studied (e.g. see \cite{DKLNTW2022, Glebov, Liu2013, Hou2018, yuan2018}).

In this paper, we mainly consider edge-critical graphs intersecting in a special vertex. For a vertex $v$ and a multi-set $\mathcal{F}$ of graphs, the suspension $\mathcal{F}+v$ of $\mathcal{F}$ is the graph obtained by connecting the vertex $v$ to all vertices of $F$ for each $F\in \mathcal{F}$.  If $\mathscr{F}=\{F\}$, then we simply write $F+v$ instead of $\mathcal{F}+v$. We call the vertex $v$ the  \emph{center vertex} of $\mathcal{F}+v$. If $\mathcal{F}$ is a multi-set consisting of $k$ copies of $K_{r-1}$, then the graph $\mathcal{F}+v$ is known as a $(k, r)$-\emph{fan}, denoted by $F_{k, r}$.  Erd\H{o}s, F\"{u}redi, Gould and Gunderson \cite{Erdosfuredi1995} first considered the Tur\'{a}n number of $F_{k,3}$ (also known as the friendship graph), and established the following result.
\begin{thm}[Erd\H{o}s, F\"{u}redi, Gould and Gunderson \cite{Erdosfuredi1995}]\label{ooo}
For every $k\geq 1$, and for every $n\geq 50k^2$,
\[
{\rm ex}(n, F_{k,3})=\left\lfloor \frac{n^2}{4} \right\rfloor+
\begin{cases}
 k^2-k      & \text {if $k$ is odd}, \\
 k^2-\frac{3}{2}k  & \text {if $k$ is even}.
\end{cases}
\]
%Further, the number of edges is best possible.
\end{thm}
For general $r$,  Chen, Gould, Pfender and Wei \cite{Chengould2003} determined ${\rm ex}(n, F_{k, r})$ for sufficiently large $n$.
\begin{thm}[Chen, Gould, Pfender and Wei \cite{Chengould2003}]\label{111}
For every $k\geq 1$ and $r\geq 2$, and for every $n\geq 16k^3r^8$,
\[
{\rm ex}(n, F_{k, r})=t_{r-1}(n)+
\begin{cases}
 k^2-k      & \text {if $k$ is odd}, \\
 k^2-\frac{3}{2}k  & \text {if $k$ is even}.
\end{cases}
\]
%Further, the number of edges is best possible.
\end{thm}
We further extend this result and determine ${\rm ex}(n ,H)$ for $H:= \mathcal{F}+v$, where $\mathcal{F}$ consists of $k$ edge-critical graphs $H_1, H_2, \ldots, H_k$ with $\chi(H_i)= r$ for each $1\le i\le k$. Let $\mathcal{G}_{n, k ,r}$ be a family of graphs, each of which is obtained from Tur\'{a}n graph $T_r(n)$ by embedding two vertex disjoint copies of $K_k$ in one partite set if $k$ is odd and embedding a graph with $2k-1$ vertices, $k^2-3k/2$ edges with maximum degree $k-1$ in one partite set if $k$ is even. Our main result is as follows.

\begin{thm}\label{th1}
Suppose that $k\geq 1$ and $r\ge 2$ are integers. Let $H_i$ be an edge-critical graph with $\chi(H_i)=r$ for each $i\in\{1,\ldots,k\}$, and let $H:= \{H_1, H_2, \cdots, H_k\}+v$.
Then, for sufficiently large $n$,
\[
{\rm ex}(n, H)= t_r(n)+ \begin{cases}
 k^2-k      & \text {if $k$ is odd}, \\
 k^2-\frac{3}{2}k  & \text {if $k$ is even}.
\end{cases}
\]
Moreover, the $\mathcal{G}_{n, k ,r}$ is the family of extremal graphs for $H$.
\end{thm}

We also obtain the following stability theorem for $H$ defined as above, extending a result of Roberts and Scott \cite{robetscott2018} on edge-critical graphs (see Lemma \ref{Scott} for more details).
\begin{thm}\label{th1-stability}
Let $f(n)=o(n^2)$ be a function and let $H$ be defined as in Theorem \ref{th1}. If $G$ is an $H$-free graph with $n$ vertices and at least $t_r(n)-f(n)$ edges, then $G$ can be made $r$-partite graph by deleting $O(n^{-1}f(n)^{3/2})$ edges.
\end{thm}

This paper is organized as follows. In the remainder of this section, we describe notations and terminologies used in our proofs. In Section \ref{Reduction}, we make a reduction of Theorem \ref{th1}, and prove it assuming Theorem \ref{th1weaker}. We prove Theorem \ref{th1weaker} in Section \ref{Lemma}. In Section 4, we prove Theorem \ref{th1-stability}.
\vspace{0.2cm}

\noindent {\bf Natation.} Let $G=(V(G), E(G))$ be a graph. We use $e(G)$ denote $|E(G)|$. We use $\delta(G)$ and $\Delta(G)$ denote the minimum and maximum degrees in $G$, respectively. For $S, T\subseteq V(G)$, we use $G[S]$ denote the graph induced by $S$. For $v\in V(G)$, let $N_S(v)$ denote the set of vertices in $S$ adjacent to $v$ and $d_S(v)=|N_S(v)|$. Let $N_S(T)=\cap_{v\in T}N_S(v)$. Let $V-S=\{v\in V : v\notin S\}$. In particular, if $S=V(G)$, then we substitute $N_G(v)$ and $d_G(v)$ for $N_{V(G)}(v)$ and $d_{V(G)}(v)$, respectively. A \emph{matching} in $G$ is a set of edges from $E(G)$, no two of which share a common vertex. The \emph{matching number} of $G$, denoted by $\nu(G)$, is the maximum number of edges in a matching in $G$. An $r$-partition of $G$ is a partition of $V(G)$ into $r$ pairwise disjoint nonempty subsets $V_1,V_2,\ldots,V_r$. For an integer $t$, let $[t]=\{1, 2, \ldots, t\}$.

\section{Reduction to $H$-free graphs with large minimum degree}\label{Reduction}
In this section, we make a reduction in preparation for the proof of Theorem \ref{th1}. We first introduce a function related to the number of edges in a graph with bounded matching number and maximum degree.

Let $G$ be a graph with its matching number $\nu(G)$ and maximum degree $\Delta(G)$. Define
\[
f(\nu, \Delta)=\max\{e(G)\,|\, \nu(G)\leq \nu, \Delta(G)\leq \Delta\}.
\]
Abbott, Hanson and Sauer \cite{Abbott1972} studied this function for $\nu=\Delta=k-1$, and proved that
$$f(k-1, k-1)=\begin{cases}
 k^2-k      & \text {if $k$ is odd}, \\
 k^2-\frac{3}{2}k  & \text {if $k$ is even}.
\end{cases}$$
The extremal graphs are graphs with $2k-1$ vertices, $k^2-3k/2$ edges with maximum degree $k-1$ if $k$ is even, or two vertex disjoint copies of $K_k$ if $k$ is odd. For general $\nu$ and $\Delta$, Chv\'{a}tal and Hanson \cite{Hanson1976} established the following theorem.
\begin{thm}[Chv\'{a}tal and Hanson \cite{Hanson1976}]\label{fnvdelta}
For every $\nu \geq 1$ and $\Delta \geq 1$,
\begin{equation}\label{f(nu_delta)}
f(\nu, \Delta)=\nu\Delta+\left\lfloor\frac{\Delta}{2}\right\rfloor \left\lfloor\frac{\nu}{\lceil\Delta/2\rceil}\right\rfloor \leq \nu\Delta + \nu.
\end{equation}
\end{thm}

\begin{definition}[Good partition]\label{defgood}
For two integers $k,r\ge2$, call a partition $V_1, V_2,\ldots,V_r$ of $G$ $k$-good, if the following properties hold for each $i \in [r]$:
\begin{enumerate}
\item[$(\mathrm{i})$] $V_i \neq \emptyset \, \, \, and \, \, \, \Delta(G[V_i])\leq k-1$,
\item[$(\mathrm{ii})$] $\sum_{j\in[r] \setminus \{i\}}\nu (G[V_j])\leq k-1$ and
\item[$(\mathrm{iii})$] $d_{V_i}(u)+\sum_{j\in[r] \setminus \{i\}} \nu \left(G[N_G(u)\cap V_j]\right)\leq k-1$ for each $u\in V_i$.
\end{enumerate}
\end{definition}

%We also need to partition the vertex set of graphs. Note that a $partition$ of a graph $G$, denoted by $(V_1,V_2,\ldots,V_r)$, is a partition $V_1,V_2,\ldots,V_r$ of $V(G)$ with $V(G)=V_1\cup V_2\cup\ldots\cup V_r$ and $V_i\cap V_j=\emptyset$ for all $1\le i<j\le r$.
%For a graph $G$ and an $r$ partition $V_1,V_2,\ldots,V_r$ of $V(G)$, we define an $r$-partite graph
%$$G[V_1, V_2, \ldots, V_r]=\left(V(G), \{v_iv_j\in E(G) \, : \, v_i\in V_i, v_j\in V_j, i\neq j \}\right).$$
Chen, Gould, Pfender and Wei \cite{Chengould2003} characterised the properties of a $k$-good partition of $G$ by showing the following lemma.
\begin{lem}[Chen, Gould, Pfender and Wei \cite{Chengould2003}]\label{main3}
Suppose that $G$ has a $k$-good partition $V_1, V_2, \ldots, V_r$. Let $G'$ be the minimal induced subgraph of $G$ such that $e(G')-\sum_{1\leq i<j \leq r}|V'_i||V'_{j}|$ is maximal, where $V_i'=V(G')\cap V_i$ for each $i\in [r]$. Then the following properties hold:
\begin{enumerate}
\item[$(\mathrm{i})$] $e(G')-\sum_{1\leq i<j \leq r}|V'_i||V'_{j}|\le f(k-1,k-1)$;
\item[$(\mathrm{ii})$] For each $i\in [r]$ and $x\in V_i'$, we have $0<d_{G'}(x)-|V(G')\backslash V_i'|\leq k-1-\sum_{j\in [r]\setminus \{i\}}\nu(G'[V_j'])$;
\item[$(\mathrm{iii})$]  If $\nu(G'[V_i'])\geq 2$ for each $i\in [r]$, then $e(G')-\sum_{1\leq i<j \leq r}|V'_i||V'_{j}|< f(k-1, k-1).$
\end{enumerate}
\end{lem}

Now, we begin with a reduction of our main theorem via the existence of a $k$-good partition, and prove Theorem \ref{th1} assuming Theorem \ref{th1weaker}. We leave the proof of Theorem \ref{th1weaker} in the next section.
\begin{thm}\label{th1weaker}
For two integers $k,r\ge2$, let $H_i$ be an edge-critical graph with $\chi(H_i)=r$ for each $i\in [k]$, and let $H := \{H_1, H_2, \ldots, H_k\}+u$. If $G$ is an $H$-free graph with $n$ vertices and $\delta(G)\geq \frac{r-1}{r}n-k$, then $G$ contains a $k$-good partition  for sufficiently large $n$.
\end{thm}

\begin{proof}[{\bf Proof of Theorem \ref{th1} given Theorem \ref{th1weaker}}]
Let $\mathcal{G}_{n, k, r}$ be the family of graphs defined in Section \ref{Intro}. We first show that $G$ is $H$-free for each $G\in \mathcal{G}_{n, k ,r}$. Otherwise, we consider an embedding of $H$ with center vertex $v$ into $G$. Without loss of generality, we may assume that $e(G[V_1])=f(k-1, k-1)$. Note that $E(H_j+v)\cap E(G[V_1])\neq \emptyset$ for each $j\in [k]$ in view of $\chi(H_j)=r$. It follows that $v\notin V_1$ as $\Delta(G[V_1])<k$ by the construction of $\mathcal{G}_{n, k, r}$. Suppose that $v\in V_s$ for some $s\in [r]\setminus \{1\}$. In this situation, we have $E(H_j+v)\cap E(G[V_1])$ are pairwise disjoint for any $j\in [k]$. This means that $\nu(G[V_1])\geq k$, a contradiction. Thus, $G$ is $H$-free and $e(G)= t_r(n)+ f(k-1, k-1)$, implying the lower bound.

In what follows, we prove that $e(G)\le t_r(n)+ f(k-1, k-1)$ for any $H$-free graph $G$ on $n$ vertices. We first show that this is true if $\delta(G)\geq \frac{r-1}{r}n-k$. By Theorem \ref{th1weaker} and Lemma \ref{main3}, there is a $k$-good partition $V_1, V_2, \ldots, V_r$ of $G$ such that
\begin{align*}
e(G)\leq \sum_{1\leq i<j \leq r}|V_i||V_{j}|+f(k-1, k-1) \leq t_r(n)+f(k-1, k-1),
\end{align*}
as desired. Next, we aim to deal with small vertices. For a graph $F$ and $f\in V(F)$, we say $f$ is a \emph{small vertex} of $F$ if $d(f)<\frac{r-1}{r}|V(F)|-k$. We first delete a small vertex in $G$. As long as there is a small vertex in the resulting graph, we delete it. We keep doing this until the remaining graph $G^*$ ($G^*$ maybe empty) has no small vertices. If $n^*:=|V(G^*)|<\sqrt{kn/ 4}$, then
\begin{align*}%\label{eG(<)}
e(G)&< e(G^*)+\sum_{i=n^*+1}^{n}\left(\frac{r-1}{r}i-k\right)\\
&< \frac{kn}{8} + \frac{r-1}{r}\frac{(n+\sqrt{kn/4})(n-\sqrt{kn/4})}{2}-\left(n-\sqrt{\frac{kn}{4}}\right)k\\
& < \frac{r-1}{r}\frac{n(n-1)}{2}\leq t_r(n),
\end{align*}
as required. Thus, we may assume that $n^*$ is sufficiently large and $\delta(G^*)\geq \frac{r-1}{r}n^*-k$. This implies that $e(G^*)\leq t_r(n^*)+f(k-1, k-1)$ as $G^*$ is also $H$-free. It follows that
\begin{align}\label{eG(>)}
e(G) &< e(G^*)+\sum\limits_{i=n^*+1}^{n}\left(\frac{r-1}{r}i-k\right)\notag\\
& \leq  t_r(n^*)+f(k-1, k-1)+\sum\limits_{i=n^*+1}^{n}\left(\frac{r-1}{r}i-k\right)\notag\\
& \leq  t_r(n) + f(k-1, k-1),
\end{align}
where the last inequality holds as $t_r(s-1)+\frac{r-1}{r}(s-1)\leq t_r(s)$. Thus, $e(G)< t_r(n)+f(k-1, k-1)$.

Now, we prove the uniqueness of the extremal graph. Let $G$ be an $H$-free graph with
\begin{equation}\label{ss}
e(G)=t_{r}(n)+f(k-1, k-1).
\end{equation}
Then $\delta (G)\geq \frac{r-1}{r}n-k$ by the above argument. It follows from Theorem \ref{th1weaker} that $G$ has a $k$-good partition $V_1,V_2,\ldots,V_r$. By Lemma \ref{main3}, there exists a minimal induced subgraph $G'$ of $G$ such that
\begin{align}\label{lll}
e(G)-\sum_{1\leq i<j \leq r}|V_i||V_{j}|=e(G')-\sum_{1\leq i<j \leq r}|V'_i||V'_{j}|=f(k-1, k-1).
\end{align}
Without loss of generality, suppose that $|V'_i|>0$ for $1\leq i\leq s$ and $|V'_i|=0$ for $s+1\leq i\leq r$. Then, for each $i\in [s]$ and $x\in V_i'$
\[
0<d_{G'}(x)-|V(G')\backslash V_i'|\leq k-1-\sum_{j\in [s]\setminus \{i\}}\nu(G'[V_j']),
\]
implying that $\nu(G'[V_i'])\ge1$ and $\sum_{j\in [s]\setminus \{i\}}\nu(G'[V_j'])\le k-2$. In addition, by \eqref{ss} and \eqref{lll}, we have
\begin{align}\label{sum_e(V_i)}
\sum_{i=1}^re(G[V_i])=\sum_{i=1}^{s}e(G'[V'_i])=f(k-1, k-1).
\end{align}

Case 1. $\sum_{j\in [s]\setminus \{i_0\}}\nu(G'[V_j'])=0$ for some $i_0\in[s]$. This implies that $V_j'=\emptyset$ for each $j\in [r]\setminus \{i_0\}$ and $G'=G[V'_{i_0}]$. Thus, $e(G[V_{i_0}])\ge e(G')=f(k-1, k-1)$. It follows from \eqref{sum_e(V_i)} that $G\cong G_{n, k ,r}\in \mathcal{G}_{n, k ,r}$.

Case 2. $1\le\sum_{j\in [s]\setminus \{i\}}\nu(G'[V_j'])\le k-2$ for each $i\in [s]$. Clearly, there exists an $i_0\in[r]$ such that $\nu(G'[V_{i_0}'])=1$; otherwise, we get a contradiction by Lemma \ref{main3}$(\mathrm{iii})$. Without loss of generality, suppose that $\nu(G'[V'_1])=1$.  Then, we have
\begin{align*}\label{EQUATION:e(G')<f(k-1,k-1)k-1}
\sum_{i=1}^se(G'[V'_i])&\leq \sum_{2\leq i\leq s}f\left(\nu(G'[V'_{i}]), k-1\right) + f(1, k-1) \notag\\
&\leq f\left(\sum_{2\leq i\leq s}\nu(G'[V'_{i}]), k-1\right)+ f(1, k-1) \notag\\
&\le f(k-2, k-1)+f(1, k-1) \notag\\
&\leq f(k-1, k-1),
\end{align*}
where the last inequality is strictly true for $k\geq 5$. This leads to a contradiction in view of \eqref{sum_e(V_i)}. It remains to consider the situation for $k\le4$.

If $k=3$, then $s=2$ and $G[V_1']\cong G[V_2']\cong K_3$.
This together with \eqref{sum_e(V_i)} yields that $G$ is a graph formed by the complete $r$-partite graph with classes $V_1,\ldots , V_r$ embedding two triangles $x_1y_1z_1$ and $x_2y_2z_2$ in $V_1$ and $V_2$, respectively. Recall that $V_1,V_2,\ldots,V_r$ is a $k$-good partition of $G$. But $d_{G[V_1]}(x_1)+\nu \left(G[N_{V_2}(x_1)]\right)=3> k-1$, a contradiction to Definition \ref{defgood}(iii). If $k=4$, then $G[V'_1]\cong K_{1, 3}$ and $\sum_{i=2}^{s}\nu(G[V'_2])=2$.
  %or $\max\,\left\{d_{V'_i}(x) : x\in V'_i \,\,{\rm for}\,\, 2\leq i \leq s\right\}=3$.
As the same argument of the case $k=3$, we can find a vertex $x\in V_1$ such that $d_{G[V_1]}(x)+\sum_{i=2}^{s}\nu \left(G[N_{V_i}(x)]\right)=5> k-1$, a contradiction to Definition \ref{defgood} (iii). Thus, we complete the proof of Theorem \ref{th1}.
%If $\nu(G'[V_{i_0}'])=0$, then
% It follows that  This together with \eqref{0<v<k-2} yields
%$\sum_{i=0}^{r-1}\nu_i\leq k-1$.
%If $\sum_{i=0}^{r-1}\nu(G[V_{i}])= k-2$, then
%$$\sum\limits_{i=0}^{r-1}e(G[V_i])\leq \sum\limits_{i=0}^{r-1}f(\nu(G[V_{i}]), k-1)\leq f\left(\sum\limits_{i=0}^{r-1}\nu(G[V_{i}]), k-1\right)=f(k-2, k-1)< f(k-1, k-1),$$
%a contradiction to \eqref{lll}.
%If $\sum_{i=0}^{r-1}\nu(G[V_{i}])= k-1$, i.e., $\nu(G[V_{i_0}])= 1$, then
%\begin{align*}
%\sum\limits_{i=0}^{r-1}e(G[V_i])&\leq \sum\limits_{\substack{0\leq i \leq r-1 \\ i\neq j_0}}f(\nu(G[V_{i}]), k-1) + f(1, k-1)\\
%&\leq f\left(\sum\limits_{\substack{0\leq i \leq r-1 \\ i\neq j_0}}\nu(G[V_{i}]), k-1\right)+ f(1, k-1)\\
%&=f(k-2, k-1)+f(1, k-1).
%\end{align*}
%Since $f(\nu, \Delta)\leq\nu\Delta+\nu,$ we have $f(k-2, k-1)+f(1, k-1)<f(k-1, k-1)$ for $k\geq 3$, a contradiction to \eqref{lll}. Thus, $\nu(G[V_{i}])\geq 2$ for each $i\in \{0, 1, \cdots, r-1\}$. By Lemma \ref{main3}, we have that
%$$\sum\limits_{i=0}^{r-1}e(G[V_i])-\left(\sum\limits_{0\leq i<i' \leq r-1}|V_i||V_{i'}|-e(G[V_0, V_1,\cdots, V_{r-1}])\right)<f(k-1, k-1),$$
%a contradiction to \eqref{lll}. This completes the proof of Theorem \ref{th1}.
\end{proof}

\section{$H$-free graphs with large minimum degree}\label{Lemma}
In this section, we give a proof of Theorem \ref{th1weaker}. We first present the following useful lemmas given by Roberts and Scott \cite{robetscott2018}.
\begin{lem}[Roberts and Scott \cite{robetscott2018}]\label{Scott}
Let $F$ be a graph with a critical edge and $\chi(F) = r + 1 \geq 3$, and let $f(n) = o(n^2)$ be a function. If $G$ is an $F$-free graph with $n$ vertices and $e(G) \geq t_r(n) - f(n)$, then $G$ can be made $r$-partite by deleting $O(n^{-1}f(n)^{3/2})$ edges.
\end{lem}
\begin{lem}[Roberts and Scott \cite{robetscott2018}]\label{t_k^*}
Let $r\geq2$ and $t\geq 1$ be integers. Suppose that the graph $G\subseteq T_r(rn)$ is $T_r(rt)$-free. Then $e(G)\leq t_r(rn)-n^2/2$ for sufficiently large $n$.
\end{lem}

We also need another easy lemma about edge-critical graphs.
\begin{lem}\label{u_iv_i}
Let $G$ be an edge-critical graph with $\chi(G)=r \ge 2$, and let $G^* = G+u$. If $v_1v_2$ is a critical edge of $G$, then both $uv_1$ and $uv_2$ are critical edges of $G^*$.
\end{lem}
\begin{proof}[\bf Proof]
By symmetry, it suffices to show that $uv_1$ is a critical edge of $G^*$.  Since $\chi(G)=r$ and $v_1v_2$ is a critical edge of $G$, there is a partition $(V_1,\ldots,V_r)$ of $G$ such that $V_i$ is an independent set of $G$ for each $i\in [r-1]$ and $V_r=\{v_1\}$. Let $G'$ be the graph obtained from $G^*$ by deleting the edge $uv_1$. Clearly, $(V_1,\ldots,V_{r-1},V_r\cup\{u\})$ is an $r$-coloring of $G'$. This means that $uv_1$ is a critical edge of $G^*$ in view of $\chi(G^*)=r+1$.
\end{proof}

Now, we are in a position to prove Theorem \ref{th1weaker}.
\begin{proof}[{\bf Proof of Theorem \ref{th1weaker}}]
Suppose that $H_j$ is an edge-critical graph with a critical edge $u_jv_j$ and $\chi(H_j)=r$ for each $j\in [k]$. Let $H := (H_1, H_2, \ldots, H_k)+u$ and $H_j^*=H_j+u$ for $j\in[k]$. By Lemma \ref{u_iv_i}, we have $uv_j$ is a critical edge of $H^*_j$. Then, by Theorem \ref{Simonovits}, there exists a constant $p_j$ (or $p^*_j$) such that $H_j$ can be embedded in $T_{r-1}((r-1)p_j) + e$ with $u_jv_j=e$ (or $H^*_j$ can be embedded in $ T_r(rp^*_j) + e$ with $uv_j=e$) for each $j\in [k]$, where $e$ is any edge inside a vertex class of $T_{r-1}((r-1)p_j)$ (or $T_r(rp^*_j)$).

Let $G$ be an $H$-free with maximum number of edges. This means that $G+e$ contains a copy of $H$ for any $e\notin E(G)$. It follows that $G$ contains a subgraph $D$ which is a copy of $H-e_0$ for some $e_0\in H$. Without loss of generality, we may assume that $e_0\in H^*_k$ and $v_0$ is the center vertex of $D$. Let $D'$ denote the subgraph which is a copy of $(H_1, \ldots, H_{k-1})+u$ in $D$. Note that $\delta(G) \geq \frac{r-1}{r} n - k$. Let $\ell=|V(D')|$. Choose a subset $S\subseteq N_G(v_0)-V(D')$ such that
\begin{equation}\label{|S|}
|S|=\frac{r-1}{r} n - k - \ell.
\end{equation}
Clearly, $G[S]$ is $H_k$-free as $G$ is $H$-free. We show that $G[S]$ is close to $T_{r-1}(|S|)$. Note that
\begin{equation}\label{delta(G[S])}
\delta(G[S]) \geq \delta (G)-(n-|S|) \geq \frac{r-2}{r}n - 2k - \ell=|S|-\frac{n}{r}-k.
\end{equation}
This implies that
\begin{align*}
e(G[S])&\geq \frac{|S|\left(|S|-\frac{n}{r}-k\right)}{2}\geq \frac{|S|(|S|+1)}{2}-\frac{|S|\left(\frac{n}{r}+k+1\right)}{2}\\
       &=\frac{r-2}{r-1}\frac{|S|(|S|+1)}{2}-\frac{(k+\ell-1)+(r-1)(k+1)}{2(r-1)}|S|.
\end{align*}
For simple, let $C_{k, \ell}=\frac{(k+\ell-1)+(r-1)(k+1)}{2(r-1)}$. Since $(1-\frac{1}{r-1})\frac{|S|(|S|+1)}{2}\geq t_{r-1}(|S|),$
we have
\begin{equation}\label{E(G[S])}
e(G[S])\geq t_{r-1}(|S|)-C_{k, \ell}|S|.
\end{equation}

For  a partition $(S_1, S_2, \ldots, S_{r-1})$ of $G[S]$, we define an $(r-1)$-partite graph
$$G_S[S_1, S_2, \ldots, S_{r-1}]=\left(S, \{v_{i}v_{i'}\in E(G) \, : \, v_{i}\in S_{i}, v_{i'}\in S_{i'}, 1\le i<i' \le r-1\}\right).$$
Now we partition $S$ into $(S_1, S_2, \ldots, S_{r-1})$ such that $e(G_S[S_1, S_2, \ldots, S_{r-1} ])$ is maximum. By Lemma \ref{Scott}, for some constant $c$
\begin{equation}\label{cn12}
\sum\limits_{1\leq i \leq r-1}e(G[S_i])\leq c |S|^{1/2} \leq cn^{1/2}.
\end{equation}
This together with \eqref{E(G[S])} implies that
\begin{equation}\label{|S_i|}
\left||S_i|-\frac{n}{r}\right|\leq \varepsilon_1 \frac{n}{r}
\end{equation}
for some $\varepsilon_1 \in(0, 10^{-6})$. Fix $i\in[r-1]$. For $x_i \in S_i$ and $i'\in [r-1]$ with $i'\neq i$, we have
\begin{equation}\label{dS_j|}
d_{S_{i'}}(x_i)\geq \delta(G[S])-\sum\limits_{ q \in [r-1]\setminus \{i, i'\}}|S_q|-cn^{\frac{1}{2}}\geq \frac{1}{r}n-2k-\ell-\frac{r-3}{r}\varepsilon_1n-cn^{\frac{1}{2}} \geq (1-\varepsilon_2)\frac{n}{r}
\end{equation}
for some $\varepsilon_2 \in(2\varepsilon_1, 10^{-5})$.

%\begin{claim}\label{eS_i=0}
%\end{claim}
%\begin{proof}
Now, we show that $E(G[S_i])$ is empty for each $i\in [r-1]$. Suppose that there exists an edge $uv\in E(G[S_1])$.  Pick $B_1 \subseteq (S_1-\{u, v\})$ with $|B_1|= \frac{n}{2r}$. By \eqref{|S_i|} and \eqref{dS_j|}, we have
$$|N_{S_i}(u)\cap N_{S_i}(v)|\geq (1-2(\varepsilon_1+\varepsilon_2))\frac{n}{r}\geq \frac{n}{2r}$$
for each $2\leq i\leq r-1$. We can pick a subset $B_i \subseteq \left(N_{S_i}(u)\cap N_{S_i}(v)\right)$ with $|B_i|=\frac{n}{2r}$ for each $2\leq i\leq r-1$. Let $B=\cup_{i\in[r-1]}B_i$. Recall that $G[S]$ is $H_k$-free. It follows from Theorem \ref{Simonovits} that $G_B\left[B_1, B_2, \ldots, B_{r-1}\right]$ is $T_{r-1}((r-1)p_k)$-free for some constant $p_k>0$. By Lemma \ref{t_k^*}, we have
$$e(G_B\left[B_1, B_2, \ldots, B_{r-1}\right]) \leq t_{r-1}\left(\frac{(r-1)n}{2r}\right)- \frac{n^2}{8r^2},$$
implying that there are at least $\frac{n^2}{8r^2}$ edges missing between vertex classes $S_i$ for $i\in [r-1]$. This together with \eqref{cn12} shows that
$$e(G[S])\leq t_{r-1}(|S|)+cn^{\frac{1}{2}}-\frac{n^2}{8r^2},$$ a contradiction to \eqref{E(G[S])}. Therefore $E(G[S_i])=\emptyset$ for each $i\in [r-1]$.
%\end{proof}

Since $e(G[S_i])=0$, we can further improve $|S_i|$ for each $i\in[r-1]$ by showing that
\begin{equation}\label{|S_i|up}
|S_i|\leq |S|-\delta(G[S]) \leq \frac{n}{r}+k
\end{equation}
in view of \eqref{|S|} and \eqref{delta(G[S])}. This further implies that
\begin{equation}\label{|S_i|blew}
|S_i|= |S|-\sum\limits_{ i' \in [r-1]\setminus \{i\}}|S_{i'}| \geq \frac{n}{r}-k(r-1)-\ell.
\end{equation}
Moreover, for $x\in S_i$ and $i'\in [r-1] \setminus \{i\}$, it follows from \eqref{delta(G[S])} and \eqref{|S_i|up} that
\begin{equation}\label{dS_jxin}
d_{S_{i'}}(x)\geq \delta(G[S])- \sum_{q \in[r-1] \setminus \{i, i'\} }|S_q| \geq \frac{n}{r}-k(r-1)-\ell.
\end{equation}

Recall that $D'$ denote the subgraph which is a copy of $(H_1, \ldots, H_{k-1})+u$ in $D$. We consider the vertices not in $S\cup V(D')$. Let $S_0=V(G)-S-V(D').$ Then
\begin{equation}\label{|s0|}
|S_0|=n-\ell-|S|= \frac{n}{r}+k.
\end{equation}
For $x_i\in S_i$ with $i \in [r-1]$, we have
\begin{equation}\label{dS_0}
d_{S_0}(x_i) \geq d_{G}(x_i)-(|S\cup V(D')|-|S_i|)\geq |S_i| \geq \frac{n}{r}-k(r-1)-\ell=|S_0|-kr-\ell.
\end{equation}
Let
\begin{equation}\label{ap}
a=kr+\ell, \, \, \,  p^*=\sum_{i \in [k]}p^*_i,
\end{equation}
and
$$S_0^*= \{ x\in S_0 : d_{S_0}(x) \geq p^*(r-1)a+k\}.$$
%The following claim shows that $|S_0^*|$ is small.
\begin{claim}\label{|S0*|}
$|S_0^*|\leq a(r-1)$.
\end{claim}
\begin{proof}
Suppose that $|S_0^*|\geq a(r-1)+1$. For each $i\in [r-1]$, let
$$S_0^i=\left\{v\in S_0^* :  d_{S_i}(v)\geq \frac{|S_i|}{a+1}\right\}.$$
Notice that $d_{S_0}(x_i) \geq |S_0|-a$ for $x_i\in S_i$ by \eqref{dS_0} and \eqref{ap}.  If $X\subseteq S_0^*$ with $|X|=a+1$, then $S_i\subseteq \cup_{x\in X}N(x)$ for $i\in[r-1]$.
This implies that $|S_0^i|\geq |S_0^*|-a$. Thus, $\left|\cap_{i \in [r-1]}S_0^i\right|\geq |S_0^*|-a(r-1)\geq 1.$
We can choose a vertex $v\in S_0^*$ such that for each $i\in [r-1]$
\begin{equation}\label{v>S_i/a+1}
d_{S_i}(v)\geq \frac{|S_i|}{a+1}.
\end{equation}
In the following, we aim to find a copy of $H$ with a center vertex $v$.

For $i\neq i'\in[r-1]$ and $x\in S_i$, recall that $d_{S_{i'}}(x)\geq n/r-k(r-1)-\ell\geq |S_i|-a$ by \eqref{|S_i|up}, \eqref{dS_jxin} and \eqref{ap}. This together with \eqref{v>S_i/a+1} shows that
\begin{equation}\label{N_S_iY}
N_{S_i}(Y\cup \{v\})\geq N_{S_i}(v)-a|Y|\geq\frac{|S_i|}{a+1}-aC_Y,
\end{equation}
where $Y\subseteq S-S_i$ and $|Y|=C_Y$ is a constant.
For sufficiently large $n$, by \eqref{N_S_iY}, we can pick
$Y_1^1\subseteq N_{S_1}(v)$
with $|Y_1^1|= p^*_1,$
$Y_i^1\subseteq N_{S_i}(v)\cap \big(\cap_{i'\in[i-1]}N_{S_i}(Y_{i'}^1)\big)$
satisfying $|Y_i^1|= p^*_1$ successively for $2\leq i \leq r-1$ .  Fix $i$, we choose
$$Y_i^j\subseteq \left(N_{S_i}(v)\cap \left(\bigcap\limits_{ i'\in [i-1]}N_{S_i}(Y_{i'}^j)\right)\right)-\bigcap\limits_{j' \in [j-1]}Y_i^{j'}$$
satisfying $|Y_i^j|= p^*_j$ successively for  $2\leq j \leq k$.
For $j\in [k]$, let
$Y_j=\cup_{ i \in [r-1]}Y_i^j$.
Then $|Y_j|=(r-1)p^*_j$. Now, we find an $H^*_j$ in $G[Y_j\cup S_0]$. It follows from  \eqref{dS_0} that
$$\left|N\left(Y_j\right)\cap N_{S_0}(v)\right|\geq p^*(r-1)a+k- p^*_j(r-1)a\geq (p^*-p^*_j)(r-1)a+k.$$
Thus, for $j\in [k]$, we can pick $x_j\in N\left(Y_j\right)\cap N_{S_0}(v)$ such that $x_1, x_2, \ldots, x_k$ are pairwise distinct. Again by \eqref{dS_0}, for $j\in[k]$, we have
$$|N_{S_0}(Y_j)|\geq |S_0|-p^*_j(r-1)a= \frac{n}{r}+k-p^*_j(r-1)a.$$
This means that we can pick
$Y_0^1\subseteq N_{S_0}\left(Y_1\right)- \{v, x_1, x_2, \ldots, x_k\},$
with $|Y_0^1|=p^*_1$, and pick
$$Y_0^j\subseteq N_{S_0}\left(Y_j\right)- \left(\bigcup\limits_{ j' \in [j-1]}Y_0^j\cup\{v, x_1, x_2, \ldots, x_k\}\right)$$
with $|Y_0^j|=p^*_j$ for $j=2, 3, \ldots, k$ successively.

For $j\in[k]$, we have chosen $x_j \in N_{S_0}(v)$ and subsets $Y_j$, $Y_0^j$. It is easy to see that $G[Y_j\cup Y_0^j]$ contains a copy of $T_r(rp^*_j)$. This together with the choice of the edge $vx_j$ and Theorem \ref{Simonovits} shows that $H^*_j=H_j+u$ can be embedded in $G_{Y_j\cup Y_0^j}[Y_0^j, Y_1^j,\ldots, Y_r^j]+vx_j$ such that $u=v$. Thus, we can get a copy of $H$ in $G$, a contradiction.
\end{proof}

Now, we consider the vertices in $S_0$ with small degree. Let $Z_0=S_0-S_0^*$, $Z_i=S_i$ for $i\in[r-1]$ and
$Z= Z_0\cup Z_1 \cup \ldots \cup Z_{r-1}$.
It follows from  Claim \ref{|S0*|} that
\begin{equation}\label{|Z_0|}
\frac{n}{r}+k=|S_0|\geq |Z_0|\geq |S_0|- a(r-1) \geq \frac{n}{r}+k-a(r-1).
\end{equation}
By \eqref{|S_i|up}, \eqref{|S_i|blew} and \eqref{ap}, we have
\begin{equation}\label{|Z_i|}
\frac{n}{r}+k\geq |Z_i|=|S_i|\geq \frac{n}{r}-a+k
\end{equation}
for $i\in [r-1]$, and then
\begin{equation}\label{V(G)-Z}
\left|V(G)-Z\right|\leq \left|V(G)-\bigcup\limits_{i=0}^{r-1}S_i\right|+a(r-1)\leq \ell+a(r-1).
\end{equation}
Recall that $\delta (G) \geq \frac{r-1}{r} n - k$ and $d_{Z_0}(x)\leq p(r-1)a+k-1$ for $x\in Z_0$. This together with \eqref{|Z_i|} and \eqref{V(G)-Z} shows that for $x\in Z_0$ and $i\in [r-1]$,
\begin{align}\label{*}
d_{Z_i}(x) &\geq d_G(x)-\left|V(G)-Z\right|- d_{Z_0}(x)-\sum\limits_{ i' \in [r-1]\setminus \{ i\} }|Z_{i'}| \notag\\
           &\geq \frac{r-1}{r}n - k - \ell-a(r-1)-(p(r-1)a+k-1)-(r-2)\left(\frac{n}{r}+k\right)\notag\\
           &\geq \frac{n}{r}-\left((p+1)(r-1)+1\right)a-p(r-1)k\notag\\
           &\geq \frac{n}{r}-2\left((p+1)(r-1)+1\right)a.
\end{align}
For every $i\in [r-1]$ and every $x\in Z_i$, by Claim \ref{|S0*|} and \eqref{dS_0}, we have
\begin{equation}\label{dZ_0}
d_{Z_0}(x)\geq d_{S_0}(x)-|S_0^*|\geq \frac{n}{r}-a+k-a(r-1)=\frac{n}{r}-ar+k.
\end{equation}

%The following claim deals with the vertices in $V(G)-Z$.
\begin{claim}\label{j(v)}
For every $x\in V(G)-Z$, there exists an $i=i(x)$ such that $d_{Z_i}(x)< k$. Moreover, such an $i$ is unique.
\end{claim}
\begin{proof}
Suppose that there exists a vertex $v\in V(G)-Z$ such that $d_{Z_i}(v)\geq k$ for each $i\in\{0\}\cup[r-1]$. Without loss of generality, let $d_{Z_0}(v)=$ $\min \{$$d_{Z_i}(v)\,:\, 0 \leq i \leq r-1 \}$. By the pigeonhole principle, we have
$d_{Z_0}(v)\leq d_{G}(v)/r$.
Thus, for $i \in [r-1]$, we have
\begin{align}\label{N_{Z_i}(v-)}
d_{Z_i}(v) &\geq d_G(v)-\left|V(G)-Z\right|-d_{Z_0}(v)-\sum_{ i' \in [r-1] \setminus \{i\}}d_{Z_{i'}}(v) \notag\\
           &\geq d_G(v)-(\ell+a(r-1))-\frac{d_{G}(v)}{r}-\sum_{i' \in [r-1] \setminus \{i\}}\left|Z_{i'}\right| \notag\\
           &\geq \frac{r-1}{r}d_G(v)-\frac{r-2}{r}n-ar \geq \frac{n}{2r^2}.
\end{align}
Now, we construct $k$ $r$-partite graphs. Recall that $|N_{Z_0}(v)|\geq k$. We can pick $k$ distinct vertices $x_1, x_2, \ldots, x_k$ in $N_{Z_0}(v)$ and choose $k$ pairwise disjoint subsets $Y_0^1$, $Y_0^2$, $\ldots,$ $Y_0^k$ in $Z_0 -\{x_1, x_2, \ldots, x_k\}$ with $|Y_0^j|= \frac{n}{4kr^2}$ for $j\in[k]$. By \eqref{|Z_i|}, \eqref{*} and \eqref{N_{Z_i}(v-)}, we have
\begin{equation}\label{n/4r^2}
\left|N_{Z_i}(v)\cap N_{Z_i}(x_j)\right|\geq d_{Z_i}(v)-(|Z_i|-d_{Z_i}(x_j))\geq \frac{n}{4r^2}.
\end{equation}
For $i\in[r-1]$, we can choose $k$ pairwise disjoint $Y_i^1$, $Y_i^2$, $\ldots$, $Y_i^k$ such that $Y_i^j\subseteq N_{Z_i}(v)\cap N_{Z_i}(x_j)$ and $|Y_i^j|=\frac{n}{4kr^2}$ for $j\in[k]$. This is possible. Since we can choose $Y_i^1 \subseteq N_{Z_i}(v)\cap N_{Z_i}(x_1)$. Suppose that $Y_i^1, \ldots, Y_i^{j'-1}$ have been chosen. Due to \eqref{n/4r^2}, we choose
$$Y_i^{j'}\subseteq (N_{Z_i}(v)\cap N_{Z_i}(x_j'))\setminus \bigcup\limits_{j\in [j'-1]}Y_i^j.$$
Let $Y^j=\cup_{i=0}^{r-1} Y_i^j$ for $j\in[k]$. Then, we obtain $k$ $r$-partite graphs
$G_{Y^j}[Y_0^j, Y_1^j,\ldots,Y_{r-1}^j]$.

Since $G$ is $H$-free, there exists $j_0\in [k]$ such that $G_{Y^{j_0}}[ Y_0^{j_0}, Y_1^{j_0}, \ldots, Y_{r-1}^{j_0}]$ is $T_{r}(rp^*_{j_0})$-free by Theorem \ref{Simonovits}. Thus, by Lemma \ref{t_k^*}, we have
$$e(G_{Y^{j_0}}[ Y_0^{j_0}, Y_1^{j_0}, \ldots, Y_{r-1}^{j_0}])\leq t_{r}\left(\frac{n}{4kr}\right)- \frac{n^2}{32k^2r^4}.$$
This means that
\begin{equation}\label{e(G[Z_0v, z_1...])}
e\left(G_{Z\cup \{v\}}\left[Z_0\cup \{v\}, Z_1,\ldots, Z_{r-1}\right]\right) \leq t_r(n)- \frac{n^2}{32k^2r^4}.
\end{equation}
On the other hand, by \eqref{dS_jxin}, \eqref{|Z_i|} and \eqref{dZ_0}, we have
\begin{align*}
e\left(G_{Z\cup \{v\}}\left[Z_0\cup \{v\}, Z_1,\ldots, Z_{r-1}\right]\right) &\geq e\left(G_Z\left[Z_0, Z_1,\ldots, Z_{r-1}\right]\right)=\sum\limits_{x
                                                           \in S}d_{Z_0}(x)+\sum\limits_{i=1}^{r-2}\sum_{\substack{i<i' \leq r-1 \\ x \in S_{i'}}} d_{S_{i}}(x)\\
                                                           &\geq \left(\frac{n}{r}-ar\right)\sum\limits_{i=1}^{r-1}i|S_i|= t_{r}(n)-\frac{2(a(r+1)-k)}{r}n,
\end{align*}
a contradiction to \eqref{e(G[Z_0v, z_1...])}.

Now, we prove the uniqueness of $i=i(x)$ for $x\in V(G)-Z$. Suppose that there exists $x\in V(G)-Z$ and $i, i'\in\{0\}\cup[r-1]$ such that both $d_{Z_{i_1}}(x)$ and $d_{Z_{i_2}}(x)$ are less than $k$. This means $Z_{i_1}\cup Z_{i_2}$ has at least $|Z_{i_1}|+|Z_{i_2}|-2k+2$ vertices that are not adjacent to $x$. Thus
$d_G(x)\leq n-\left(|Z_{i_1}|+|Z_{i_2}|-2k\right)< (1-1/r)n-k$ in view of \eqref{|Z_i|}, a contradiction.
\end{proof}

By Claim \ref{j(v)}, for each $x\in V(G)-Z$, there is a unique $i=i(x)$ such that $d_{Z_i}(x)< k$. We can put $x$ in $Z_{i(x)}$. Then, we get an $r$-partition $(V_0, V_1, \ldots, V_{r-1})$ of $G$ with $Z_i\subseteq V_i$ for $i\in\{0\}\cup[r-1]$.
For $x\in V(G)$, we consider the degree of $x$ in $V_i$ with $x \notin V_i$. Suppose first that $x\in V_{i(x)}$ for $x\in V(G)-Z$. For $0\leq i \leq r-1$ with $i\neq i(x)$, by \eqref{|Z_0|}, \eqref{|Z_i|} and \eqref{V(G)-Z},
\begin{equation}\label{d_V_j(x)j(x)}
d_{V_i}(x)\geq d_G(x)-d_{Z_{i(x)}}(x)-|V(G)-Z|-\sum_{\substack{0\leq i'\leq r-1 \\ i'\notin \{i, i(x)\}}}|Z_{i'}|\geq \frac{n}{r}-ar-k.
\end{equation}
Then, we bound $d_{V_{i'}}(x)$ for $x\in V_{i}$ and $i\neq i'$. For $i\in\{0\}\cup[r-1]$, it follows from \eqref{|Z_0|} and \eqref{|Z_i|} that
\begin{equation}\label{|V_i|}
\frac{n}{r}+k +\ell + a(r-1)\geq |Z_i|+|V(G)-Z| \geq |V_i|\geq |Z_i|\geq \frac{n}{r}+k-a(r-1).
\end{equation}
Let $x\in V_i$ and $i'\neq i$, $0 \leq i' \leq r-1$. Combining \eqref{*}, \eqref{dZ_0} and \eqref{d_V_j(x)j(x)},
\begin{equation}\label{d_V_j(x)inoteqj}
d_{V_{i'}}(x)\geq d_{Z_{i'}}(x)\geq \frac{n}{r}-2\left((p+1)(r-1)+1\right)a.
\end{equation}
Let $b_1=k +\ell + a(r-1)$ and $b_2=2\left((p+1)(r-1)+1\right)a$. Fixing $i\in\{0\}\cup[r-1]$, \eqref{|V_i|} and \eqref{d_V_j(x)inoteqj} can be reduced to
\begin{equation}\label{|V_i|jian}
\frac{n}{r}+b_1 \geq |V_i|\geq \frac{n}{r}+k-a(r-1)
\end{equation}
and
\begin{equation}\label{d_V_j(v_i)}
d_{V_{i'}}(x)\geq \frac{n}{r}-b_2 \geq |V_{i'}|-(b_1+b_2)
\end{equation}
for $x\in V_i$ and $i'\neq i$, $0 \leq i' \leq r-1$. By \eqref{|V_i|jian} and \eqref{d_V_j(v_i)}, we have
\begin{align}\label{**}
e\left(G\left[V_0, V_1, \ldots, V_{r-1}\right]\right)&=\sum\limits_{i=0}^{r-2}\sum_{\substack{i<i' \leq r-1 \\ x \in V_{i'}}} d_{V_i}(x)\geq\left(\frac{n}{r}-b_2\right)\sum\limits_{i=1}^{r-1}i|V_i|\notag\\
&= t_{r}(n)-\frac{2(a(r-1)-k+b_2)}{r}n.
\end{align}
In what follows, we prove that $(V_0, V_1, \ldots, V_{r-1})$ is a $k$-good partition of $G$.

First, we show that $(V_0$, $V_1, \ldots, V_{r-1})$ satisfies $(1)$ of Definition \ref{defgood}. Note that $V_i\neq \emptyset$ for $0\leq i \leq r-1$ by \eqref{|V_i|}. If $\Delta(G[V_i])\geq k$ for some $i\in \{0, 1, \ldots, r-1\}$, then we can choose $x\in V_i$ and $x_1, x_2, \ldots, x_k$ in $N_{V_i}(x)$. As in the proof of Claim \ref{j(v)}, we can find $k$ $r$-partite graphs and one of them is $T_r(rp^*_j)$-free for some $j\in [k]$ by Theorem \ref{Simonovits}. Thus, by Lemma \ref{t_k^*}, we have
\begin{equation}
e\left(G\left[V_0, V_1, \ldots, V_{r-1}\right]\right)\leq t_r(n)-\varepsilon n^2
\end{equation}
for some $\varepsilon>0$, a contradiction to \eqref{**}.

Then, we show that $(V_0$, $V_1, ldots, V_{r-1})$ satisfies $(2)$ of Definition \ref{defgood}. Otherwise, by symmetry, suppose that
$\sum_{i \in [r-1] }\nu(G[V_i])\geq k$.
Let $x_1y_1, x_2y_2, \ldots, x_ky_k$ be the matching $M$ of $G$ with $x_jy_j\in V_{i(x_jy_j)}$ for some $i(x_jy_j)\in [r-1]$. We use $M$ to find $k$ $(r-1)$-partite graphs. By \eqref{|V_i|jian} and \eqref{d_V_j(v_i)}, we have
$$\left| \bigcap\limits_{j\in[k]}\left(N_{V_0}(x_j)\cap N_{V_0}(y_j)\right) \right|\geq |V_0|-2k(b_1+b_2)\geq 0.$$
Choose a vertex $v\in \cap_{j\in[k]}(N_{V_{i_0}}(x_j)\cap N_{V_{i_0}}(x_j))$. For $j\in [k]$ and $q\in [r-1]$ with $q\neq i(x_jy_j)$, let
$X^j_q=  N_{V_q}(v)\cap N_{V_q}(u_j) \cap N_{V_q}(v_j).$
Clearly, $|X^j_q|\geq n/r-3(b_1+b_2)\geq  n/(2r)+2k$ by \eqref{d_V_j(v_i)}.
We first choose $r-2$ subsets $Y_q^1$ with $q\in[r-1]$ and $q\neq i(x_1y_1)$ such that $Y_q^1 \subseteq X_q^1$ and $|Y_q^1|=\frac{n}{2kr}$. Then, for $j\in \{2, 3, \ldots, k\}$, we choose $r-2$ subsets $Y_q^j$ with $q\in[r-1]$ and $q\neq i(x_jy_j)$ such that $Y_q^j \subseteq X^j_q \backslash \bigcup\limits_{s\in[j-1]}Y_q^s$ and $\left|Y^j_q\right|= \frac{n}{2kr}$. Finally, we choose
$$Y^j_{i(x_jy_j)}\subseteq V_{i(x_jy_j)}(v)\backslash \big(\{x_j, y_j\}\cup \big( \bigcup\limits_{ s \in [k]\setminus \{j\}}Y^s_{i(x_jy_j)}\big)\big)$$
with $|Y^j_{i(x_jy_j)}|=\frac{n}{2kr}$ for $j\in [k]$. Let $Y^j = \cup_{i\in[r-1]} Y^j_i$ for $j\in[k]$. Then, we obtain $k$ $(r-1)$-partite graphs
$G_j=G_{Y^j}[Y^j_1, Y^j_2, \ldots, Y^j_{r-1}]$.
Since $G$ is $H$-free, there exists some $j_0\in [k]$ such that $G_{j_0}$ is $T_{r}(rp_{j_0})$-free by Theorem \ref{Simonovits}. Thus, by Lemma \ref{t_k^*}, we have
$$e(G_{j_0})\leq t_{r}\left(\frac{n}{2k}\right)- \frac{n^2}{8k^2r^2}.$$
This means that
$$e\left(G[V_0, V_1,\ldots, V_{r-1}] \right)\leq t_{r}(n)-\frac{n^2}{8k^2r^2},$$
a contradiction to \eqref{**}.

In the end, we show that $(V_0, V_1, \ldots, V_{r-1})$ satisfies $(3)$ of Definition \ref{defgood}. Otherwise, by symmetry, suppose that there exists $v\in V_0$ such that
$d_{V_0}(v)+\sum_{i \in [r-1] } \nu \left(G[N_{V_i}(v)]\right)\geq k.$
Thus, we can pick $z_1, z_2, \ldots, z_s$ in $N_{V_0}(v)$ and $(k-s)$-matching  $x_{s+1}y_{s+1}$, $x_{s+2}y_{y+2}$, $\ldots$, $x_{k}y_{k}$ in $\cup_{i \in [r-1]}G[N_{V_i}(v)]$ such that $x_j$ and $y_j$ are in the same vertex class $V_{i(x_jy_j)}$ for $s+1 \leq j \leq k$.
As the same methods used to verify $(1)$ and $(2)$ of Definition \ref{defgood}, we can show that
$e(G[V_0, V_1, \ldots, V_{r-1}])\leq t_r(n)-\varepsilon n^2$
for some $\varepsilon >0$ by finding $s$ $k$-partite graphs $Y_1, Y_2, \ldots, Y_s$ with $v_j\in V(Y_j)$, and $(k-s)$ $(r-1)$-partite graphs $Y_{k-s+1}, \ldots, Y_k$ with $x_j,y_j\in V(Y_j)$, a contradiction to \eqref{**}.
Thus, we conclude that $(V_0, V_1, \ldots, V_{r-1})$ is a $k$-good partition of $G$, completing the proof of Theorem \ref{th1weaker}.
\end{proof}

\section{Stability results for the suspension of edge-critical graphs}
In this section, we prove Theorem \ref{th1-stability}. We first present the following fundamental result, the Erd\H{o}s-Simonovits Stability Theorem, used in our proof.
\begin{thm}[Erd\H{o}s-Simonovits Stability Theorem \cite{Erdos1966, Simonovits1974}]\label{erdossimonovits}
Let $r\geq 2$ and suppose that $F$ is a graph with $\chi (F) =r + 1$. If $G$ is an $F$-free graph with $e(G)\geq t_r(n)-o(n^2)$, then $G$ can be formed from $T_r(n)$ by adding and deleting $o(n^2)$ edges.
\end{thm}
Now, we prove a weaker result as a stepping stone to Theorem \ref{th1-stability}.
\begin{lem}\label{Ofn}
Let $f(n)=o(n^2)$ be a function. Suppose that $G$ is an $H$-free graph on $n$ vertices such that $e(G)\geq t_{r}(n)-f(n).$ Then $G$ can be made $r$-partite by deleting $O(f(n))$ edges.
\end{lem}
\begin{proof}[{\bf Proof}]
Let $\delta, \varepsilon, \eta \in [0, 10^{-4}r^{-6}k^{-4}]$ with $\varepsilon \leq \eta^2 /5$, $\gamma =100r^{2}k(\varepsilon^{1/2}+ \delta + \eta)$ and $\theta> 100r^3k(\varepsilon^{1/2}+ \delta + \eta)$. Suppose that $G$ is an $H$-free graph on $n$ vertices with $e(G)\geq t_{r}(n)-f(n)$. By Theorem \ref{erdossimonovits}, there exists some $N_0$ such that $G$ can be formed from $T_{r}(n)$ by adding and deleting $\varepsilon n^2$ edges for $n\geq N_0$. Suppose that $n \geq 2 N_0$. We say that $f$ is a \emph{``smaller" vertex} of $F$ if $d(f)\leq(1-\delta)\frac{r-1}{r}|V(F)|-1$. We first delete a ``smaller" vertex of $G$. As long as there is a ``smaller" vertex in the resulting graph, we delete it. We keep doing this until the remaining graph $G'$ ($G'$ maybe empty) has no such vertices. Let $L=V(G)-V(G')$.
\begin{claim}\label{CLAIM:|L|}
$|L|\leq (k^2+f(n))\left(\frac{2r}{\delta (r-1)}n^{-1} \right)=o(n)$.
\end{claim}
\begin{proof}
Note that $G'$ is $H$-free. Thus, $e(G')\leq t_r(n-|L|)+k^2$ by Theorem \ref{th1 }. %We first show that $|L|\leq n/2$.  Suppose, to the contrary, $|L|> n/2$. Then,
%\begin{align}
%e(G)\leq t_{r}(n/2)+
%\end{align}
It follows that
%\begin{align}
%e(G) &\leq e(G')+|L|\left((1-\delta)\frac{r-1}{r}n  \right) \notag\\
%      &\geq t_r(n)-f(n)-|L|\left((1-\delta)\frac{r-1}{r}n\right) \notag\\
%      &= t_r(n)-\frac{r-1}{r}n|L|-f(n)+\frac{r-1}{r}\delta n|L| \notag\\
%      &\geq t_r(n-|L|)-|L|-f(n)+\frac{r-1}{r}\delta n|L|
%\end{align}
\begin{align}
e(G) %&\leq e(G')+\sum\limits_{i=|V(G')|+1}^{n}\left((1-\delta)\frac{r-1}{r}i-1\right) \notag \\
     &\leq t_r(n-|L|)+k^2+\sum\limits_{i=|V(G')|+1}^{n}\left((1-\delta)\frac{r-1}{r}i-1\right) \notag\\
     &=t_r(n-|L|)+k^2+\sum\limits_{i=|V(G')|+1}^{n}\left(\frac{r-1}{r}i-1\right)-\sum\limits_{i=|V(G')|+1}^{n}\left(\delta\frac{r-1}{r}i\right) \notag \\
     &\leq t_r(n)+k^2-\sum\limits_{i=|V(G')|+1}^{n}\left(\delta\frac{r-1}{r}i\right)  \label{EQ:MAIN2_e(G<)(1)}\\
     &=t_r(n)+k^2-\delta\frac{r-1}{r}\left(n|L|-\frac{|L|^2}{2}+\frac{|L|}{2}\right)\label{EQ:MAIN2_e(G<)(2)}.
\end{align}
Note that the function $f(x)=(1/2+n)x-x^2/2$ is strictly monotone increasing on $[0, n]$. If $|L|\geq n/2$, then $f(|L|)\geq f(n/2)$. It follows from \eqref{EQ:MAIN2_e(G<)(2)} that
$$e(G)<t_r(n)+k^2-\delta\frac{r-1}{r}f(n/2)=t_r(n)+k^2-\delta\frac{r-1}{r}(3n^3/8+n/4).$$
This is a contradiction to $e(G)\geq t_r(n)-f(n)$ for sufficiently large $n$. Thus, $|L|< n/2$. In view of \eqref{EQ:MAIN2_e(G<)(1)}
$$e(G)\leq t_r(n)+k^2-|L|\left(\delta\frac{r-1}{r}\frac{n}{2}\right).$$
This together with $e(G)\geq t_r(n)-f(n)$ yields that $|L|\leq (k^2+f(n))\left(\frac{2r}{\delta (r-1)}n^{-1} \right)$.
%Consider an arbitrary subset $B\subset L$ with $|B|\leq \frac{\delta n}{2}$. Let $G'=G-B$. Note that $G'$ is $\mathcal{H}$-free. By Theorem \ref{th1},
%\begin{equation}\label{Bbelow}
%e(G')\leq t_{r+1}(n-|B|)+k^2.
%\end{equation}
%on the other hand, we have that
%\begin{align*}
%e(G')&\geq e(G)-|B|\left(\frac{(1-\delta)r}{r+1}n\right) \\
%     &\geq t_{r+1}(n)-f(n)-|B|\left(\frac{(1-\delta)r}{r+1}n\right) \\
%     &\geq t_{r+1}(n-|B|)-f(n)+\frac{r}{r+1}(n-|B|)|B|-|B|\left(\frac{(1-\delta)r}{r+1}n\right) \\
%     &= t_{r+1}(n-|B|)-f(n)+\frac{r}{r+1}(\delta n|B|-|B|^2) \\
%     &\geq t_{r+1}(n-|B|)-f(n)+\frac{\delta rn}{2(r+1)}|B|.
%\end{align*}
%This together with equation \eqref{Bbelow} shows that
%$$|B|\leq (k^2+f(n))\left(\frac{2(r+1)}{\delta r}n^{-1} \right).$$
%Recall that $f(n)=o(n^2)$ and $k$ ia a fix constant. Thus, we have $|B|=o(n)<\frac{\delta n}{4}$. By the arbitrariness of $B$, we have $|L|<\frac{\delta n}{2}$. Furthermore,
%\begin{equation}\label{Lupbound}
%|L|\leq (k^2+f(n))\left(\frac{2(r+1)}{\delta r}n^{-1} \right).
%\end{equation}
\end{proof}

By Claim \ref{CLAIM:|L|}, we have
\begin{equation}\label{eG*}
e(G')\geq e(G)- |L|\left((1-\delta)\frac{r}{r+1}n\right)=e(G)-O(f(n))\geq t_{r+1}(n)-O(f(n))
\end{equation}
for sufficiently large $n$. So, it suffices to show that $G'$ can be formed from an $r$-partite graph by deleting at most $(r+1)f(k-1, k-1)=o(n)$ edges.

%$d_{G'}(v)\geq (1-\delta)\frac{r-1}{r}|V(G')|$ for each $v\in V(G')$ by the definition of $L$.

Note that $|V(G')|= n-|L|= (1- o(1))n\geq N_1$. Thus, we can partite $V(G')$ to $V'_1, V'_2, \ldots, V'_{r}$ such that there are at most $\varepsilon n^2$ edges within the vertex classes by Theorem \ref{erdossimonovits} and \eqref{eG*}. This implies that there are at most $3\varepsilon n^2/2$ edges missing between vertex classes $V'_1, V'_2, \ldots, V'_{r}$, i.e.,
\begin{equation}\label{eV1V2}
e(G'[V'_1, V'_2, \ldots, V'_{r}])\geq t_{r}(n)-\frac{3}{2}\varepsilon n^2 \geq t_{r}(|V(G')|)-\frac{3}{2}\varepsilon n^2.
\end{equation}
For every $v\in V(G')$, we have
\begin{align}\label{EQ:d_G'(v)}
d_{G'}(v)&\geq (1-\delta)\frac{r-1}{r}(n-|L|) \notag\\
         &\geq (1-\delta)\frac{r-1}{r}n-(1-\delta)\frac{r-1}{r}(k^2+f(n))\left(\frac{2r}{\delta r-1}n^{-1}\right)  \notag\\
         &= (1-\delta)\frac{r-1}{r}n-o(n)\geq (1-2\delta)\frac{r-1}{r}n.
\end{align}
Suppose that there exists some vertex $v$ having at least $k\eta n$ neighbours in every $V'_i$ for $i\in [r]$. For each $i\in [r]$, pick $F_i^1\subset N_{V'_i}(v)$
with $|F_i^1|=\eta n$, and pick
$F^j_i\subset N_{V'_i}(v) \setminus \cup_{1\leq j'<j}F_i^{j'}$
for each $j\in [k]\setminus \{1\}$ with $|F_i^j|=\eta n$. Let $Q_j=\cup_{i=1}^r F_i^j$.
Recall that $G'$ is $H$-free. Thus, there exists $j_0$ such that $G'_{Q_{j_0}}[F^{j_0}_1, F^{j_0}_2, \ldots, F_{r}^{j_0}]$ is $T_{r}(rp_{j_0})$-free. By Lemma \ref{t_k^*}, we have
$$e(G'_{Q_{j_0}}[F^{j_0}_1, F^{j_0}_2, \ldots, F_{r}^{j_0}])\leq t_{r}(r\eta n)-\frac{(\eta n)^2}{2},$$ implying that there are at least $(\eta n)^2/2$ edges missing between $V'_1, V'_2, \ldots, V'_{r}$, a contradiction to \eqref{eV1V2}.

Without loss of generality, we can suppose that every vertex in $G'$ has at most $k\eta n$ neighbours inside its own vertex class. This together with \eqref{EQ:d_G'(v)} shows that for each $v\in V'_i$ with $i\in [r]$
\begin{equation}\label{dV3i}
d_{V(G')\setminus V'_{i}}(v)\geq (1-2\delta)\frac{r-1}{r}n - \eta n.
\end{equation}
By \eqref{eV1V2}, it is easy to see that
\begin{equation}\label{2|V_i|}
|V'_i|\geq \frac{|V(G')|}{r}-(r-1)\sqrt{\frac{3}{2}\varepsilon}n \geq (1-2r\varepsilon^{1/2})\frac{n}{r}
\end{equation}
for each $i\in [r]$. On the other hand
\begin{equation}\label{2|V_i|up}
|V'_i|\leq n- (r-1)(1-2r\varepsilon^{1/2})\frac{n}{r}=(1+2r^2\varepsilon^{1/2} )\frac{n}{r}.
\end{equation}
If $e(G'[V_i])\geq f(k-1,k-1)$ for some $i\in[r]$, say $i=1$, then we can find a $k$-matching in $V_1'$ or find a vertex $v\in V'_1$ such that $d_{V'_1}(v)\geq k$. In the following, we divide our proof into the following
two cases.

Case 1. $V'_1$ contains a vertex $v$ such that $d_{V'_1}(v)\geq k$. Suppose that $v_1, v_2, \ldots, v_k$ are $k$ distinct vertices in $N_{V'_1}(v)$. By \eqref{dV3i} and \eqref{2|V_i|}, we have
\begin{align*}
|N_{V'_j}(v)\cap N_{V'_j}(v_i)|&\geq |V'_j|-2\left((n-|V'_1|)-\left((1-2\delta)\frac{r-1}{r}n - \eta n\right)\right)\\
&\geq \frac{n}{r}-\left(\frac{6r\varepsilon^{1/2}+4\delta(r-1)}{r}n+2\eta n\right)\\
&\geq (1-\gamma)\frac{n}{r}.
\end{align*}
Pick $F_1^1\subseteq V'_1\setminus \{v, v_1, v_2, \ldots, v_k\}$ with $|F_1^1|= (1-\gamma)\frac{n}{2kr}$, and pick $F_i^1\subseteq N_{V'_i}(v)\cap N_{V'_i}(v_1)$ with $|F_i^1|= (1-\gamma)\frac{n}{2kr}$ for $i\in [r]\setminus \{1\}$. Then, pick
$$F_1^j\subseteq V'_1 \setminus \left(\bigcup\limits_{j'< j}F_1^{j'}\cup \{v, v_1, v_2, \ldots, v_k\}\right)$$
and
$F_i^j\subseteq V'_i \setminus \cup_{j'< j}F_i^{j'}$
with $|F_i^j|= (1-\gamma)\frac{n}{2kr}$ for each $i\in [r]\setminus \{1\}$ and each $j\in [k]\setminus \{1\}$. Let $Q_j=\cup_{i=1}^{r}F_j^i$.
We obtain $k$ $r$-partite graphs $G_{Q_j}[F_1^j, F_2^j, \ldots, F_{r}^j]$ for $j\in [k]$. Since $G$ is $H$-free, there exists $j_0$ such that $G'_{Q_{j_0}}[ F_1^{j_0}, F_2^{j_0}, \ldots, F_{r}^{j_0}]$ is $T_{r}(rp_{j_0})$-free. Thus, by Lemma \ref{t_k^*}
$$e\left(G'_{Q_{j_0}}[ F_1^{j_0}, F_2^{j_0}, \ldots, F_{r}^{j_0}]\right)\leq t_{r}\left(\frac{(1-\gamma)n}{2k}\right)- \frac{(1-\gamma)^2 n^2}{8k^2r^2},$$
a contradiction to \eqref{eV1V2}.

Case 2. $V'_1$ contains a $k$-matching, say $\{u_1v_1, u_2v_2, \ldots, u_kv_k\}$. For $i\in [r]\setminus \{1\}$, by \eqref{dV3i} and \eqref{2|V_i|}
\begin{align}\label{EQ:uvjoin}
\left|\bigcap\limits_{j=1}^k (N_{V'_{r}}(u_j)\cap N_{V'_{r}}(v_j))\right|&\geq |V'_r|-2k\left((n-|V_1|)-\left((1-2\delta)\frac{r}{r+1}n - \eta n\right)\right)\notag\\
&\geq \frac{n}{r}-\left(\frac{(4k+2)r\varepsilon^{1/2}+4k\delta(r-1)}{r}n+2k\eta n\right)\notag\\
&\geq (1-k\gamma)\frac{n}{r}.
\end{align}
%Fix $i\in [r]$ and $x\in V_i'$. By \eqref{dV3i} and \eqref{2|V_i|up}, we have
%\begin{align}\label{EQ:d_{V'_j}(x)}
%d_{V'_j}(x)\geq d_{V(G')\setminus V'_{r}}(x)- \sum_{\substack{1\leq q\leq r-1 \\ q\neq j}} |V'_q|\geq (1-\theta)\frac{n}{r}
%\end{align}
%for $j\in[r]\setminus \{i\}$.
Note that $(1-k\gamma)\frac{n}{r}>1$. Choose $v\in \cap_{i=1}^k (N_{V'_{r}}(u_i)\cap N_{V'_{r}}(v_i))$. By \eqref{dV3i} and \eqref{2|V_i|up}, we have
$$d_{V'_j}(v)\geq d_{V(G')\setminus V'_{r}}(v)- \sum_{q\in[r-1]\backslash\{j\}} |V'_q|\geq (1-\theta)\frac{n}{r}$$
for $j\in[r-1]$. This together with \eqref{2|V_i|up} and \eqref{EQ:uvjoin} shows that for $i\in [r-1]\setminus \{1\}$
\begin{align*}
\left|N_{V'_i}(v)\cap \left(\bigcap\limits_{j=1}^k (N_{V'_{r}}(u_j)\cap N_{V'_{r}}(v_j))\right)\right| &\geq \left|\bigcap\limits_{j=1}^k (N_{V'_{r}}(u_j)\cap N_{V'_{r}}(v_j))\right|-(|V'_i|-d_{V'_i}(v))\\
&\geq (1-k\gamma)\frac{n}{r}-\left((1+2r^2\varepsilon^{1/2})\frac{n}{r}-(1-\theta)\frac{n}{r}\right)\\
&\geq(1-10k\theta)\frac{n}{r}.
\end{align*}
Pick $F_1^1\subseteq N_{V'_1}(v)\setminus \{u_1, v_1, u_2, v_2, \ldots, u_k, v_k\}$ with $|F_1^1|= (1-10k\theta)\frac{n}{2kr}$, and pick $F_i^1\subseteq N_{V'_i}(v)\cap (\bigcap\limits_{j=1}^k (N_{V'_{r}}(u_j)\cap N_{V'_{r}}(v_j)))$ with $|F_i^1|= (1-10k\theta)\frac{n}{2kr}$ for $i\in [r]\setminus \{1\}$. Then, pick
$$F_1^j\subseteq N_{V'_1}(v) \setminus \left(\bigcup\limits_{j'< j}F_1^{j'}\cup \{u_1, v_1, u_2, v_2, \ldots, u_k, v_k\}\right)$$
with $|F_1^j|= (1-10k\theta)\frac{n}{2kr}$, and $F_i^j\subseteq V'_i \setminus \cup_{j'< j}F_i^{j'}$ with $|F_i^j|= (1-10k\theta)\frac{n}{2kr}$ for each $i\in [r]\setminus \{1\}$ and each $j\in [k]\setminus \{1\}$. Let $Q_j=F_1^j\cup F_2^j \cup \ldots \cup F_{r-1}^j$ for each $j\in [k]$.
Since $G$ is $H$-free, there exists $j_0$ such that $G'_{Q_{j_0}}[F_1^{j_0}, F_2^{j_0}, \ldots, F_{r-1}^{j_0}]$ is $T_{r-1}(p^*_{j_0}(r-1))$-free. Thus, by Lemma \ref{t_k^*}, we have
$$e\left(G'_{Q_{j_0}}[F_1^{j_0}, F_2^{j_0}, \ldots, F_{r-1}^{j_0}]\right)\leq t_{r-1}\left(\frac{(1-10k\theta)(r-1)n}{2kr}\right)- \frac{(1-10k\theta)^2 n^2}{8k^2r^2}.$$
This is a contradiction to \eqref{eV1V2}.

Thus, $e(G'[V'_i])\leq f(k-1,k-1)$ for each $i\in[r]$. This implies that $G'$ can be formed from $r$-partite graph by deleting at most $rf(k-1, k-1)=o(n)$ edges, completing the proof of Lemma \ref{Ofn}.
\end{proof}

\begin{proof}[\bf{Proof of Theorem \ref{th1-stability}}]
We can take a partition $(V_1, V_2, \ldots, V_{r})$ of $V(G)$ which minimises the number of edges inside vertex classes. By Lemma \ref{Ofn}, there are $c'_0f(n)=O(f(n))$ edges within vertex classes and at most $c'_1f(n)=O(f(n))$ edges between vertex classes are not present in $G$.
\begin{claim}\label{CLAIM:D_(V)=o(F)}
For $i\in [r]$ and $v\in V_i$, $d_{V_i}(v)=O(f(n)^{1/2})$.
\end{claim}
\begin{proof}
Suppose that there exists some $i_0$ such that $d_{V_i}(v)> c_0 f(n)^{1/2}$ for some $v\in V_i$. Without loss of generality, let $i_0=1$ and $v\in V_1$ such that $d_{V_1}(v)> 2k(c'_1 f(n))^{1/2}$. Then, $|N_{V_i}(v)|> 2k(c'_1 f(n))^{1/2}$
%\begin{equation}\label{th2dV1v}
%|N_{V_i}(v)|> 2k(c'_1 f(n))^{1/2}
%\end{equation}
for each $i\in [r]$. Pick $F_i^1\subseteq N_{V_i}(v)$ and $F_i^j\subseteq N_{V_i}(v)\setminus \cup_{j'<j}F_i^{j'}$ with $|F_i^j|=2(c'_1 f(n))^{1/2}$ for each $i\in [r]$ and each $j\in [k]$. Let $Q_j=\cup_{i=1}^r F_i^j$. Since $G$ is $H$-free, there exists $j_0$ such that $G[F_1^{j_0}, F_2^{j_0}, \ldots, F_{r}^{j_0}]$ is $T_{r}(p_{j_0}r)$-free. Thus, by Lemma \ref{t_k^*}, we have
$$e\left(G[F_1^{j_0}, F_2^{j_0}, \ldots, F_{r}^{j_0}]\right)\leq t_{r}\left(2r(c'_1 f(n))^{1/2}\right)- 2rc_1'f(n).$$
implying that there are at least $2rc_1'f(n)$ edges missing between $V_1, V_2, \ldots, V_{r}$, a contradiction.
\end{proof}
By Claim \ref{CLAIM:D_(V)=o(F)} and the proof of Lemma \ref{Ofn}, $G$ can be made $r$-partite by deleting at most $O\left(f(n)^{\frac{1}{2}}\right)|L|+r f(k-1,k-1)=O\left(n^{-1}f(n)^{\frac{3}{2}}\right)$ edges. This completes the proof.
\end{proof}

\end{document}